\documentclass{amsart}
\usepackage{amsmath,amsfonts,amssymb,amsthm}

\usepackage[T2A]{fontenc}

\usepackage{stmaryrd}
\usepackage{xcolor}
\usepackage{mathrsfs}

\numberwithin{equation}{section}

\newtheorem{Th}{Theorem}[section]
\newtheorem{Cor}[Th]{Corollary}
\newtheorem{Lem}[Th]{Lemma}
\newtheorem{Prop}[Th]{Proposition}
\newtheorem{Claim}[Th]{Claim}

\def\cB{\mathscr B}
\def\cC{\mathscr C}
\def\cD{\mathscr D}
\def\cL{\mathscr L}

\def\cF{\mathcal F}

\def\cX{\mathscr X}
\def\cY{\mathscr Y}
\def\cZ{\mathscr Z}

\def\inn#1{\mathrm{Inn}(#1)}
\def\aut#1{\mathrm{Aut}(#1)}
\def\out#1{\mathrm{Out}(#1)}
\def\ia#1{\mathrm{IA}(#1)}

\def\sym#1{\mathrm{Sym}(#1)}

\def\Mod#1{\,(\operatorname{mod} #1)}
\def\av#1{\overline{#1}}
\def\str#1{\langle #1 \rangle}

\def\id{\mathrm{id}}
\def\To{\Rightarrow}
\def\eps{\varepsilon}
\def\f{\varphi}
\def\s{\sigma}

\def\le{\leqslant}
\def\ge{\geqslant}

\def\sle{\subseteq}

\def\N{\mathbf N}
\def\Z{\mathbf Z}

\def\iat{{\rm IA}}

\def\rank{\mathrm{rank}}

\def\inv{{}^{-1}}

\def\cF{\mathcal F}
\def\cG{\mathcal G}
\def\cH{\mathcal H}
\def\cM{\mathcal M}

\def\At{{\rm A}}
\def\Za{\mathrm{Z}}
\def\Zo{\mathrm{Z}_{\text{\rm o}}}

\newtheorem{innercustomthm}{Theorem}
\newenvironment{customthm}[1]
  {\renewcommand\theinnercustomthm{#1}\innercustomthm}
  {\endinnercustomthm}

\begin{document}

\title[Complete outer automorphism groups]
 {Complete outer automorphism groups of free nilpotent groups}

\author{Vladimir A. Tolstykh}
\address{Vladimir A. Tolstykh\\
Department of Mathematics and Computer Science \\
Istanbul Arel University \\
34537 Tepekent - B{U}y{u}k{c}ekmece \\
Istanbul \\
Turkey}
\email{vladimirtolstykh@arel.edu.tr}


\maketitle

\begin{abstract}
We prove that the outer automorphism group $\out N$
of an infinitely generated free nilpotent group $N$
of class two is complete.
\end{abstract}

\section{Introduction}

The general conjecture proposed by Baumslag (1970s)
stated that a ``sufficiently symmetric group was
likely to have a very short (automorphism) tower'', and specifically
it would be the case for relatively free
groups \cite{DFo_shrt}.
The conjecture has been confirmed
in the ``sharpest sense'' for nonabelian
(absolutely) free groups, since, as it has been shown in
\cite{DFo, To_Towers},
the automorphism group $\aut F$ of any nonabelian
free group $F$ is complete. Recall that a group
$G$ is said to be {\it complete} if $G$ is centerless
and all automorphism of $G$ are inner, and
so $G \cong \inn G =\aut G.$

Baumslag's conjecture has been also confirmed
for free nilpotent groups \cite{DFo, Kass, To_2step,
To_Aut(N)} and
for free solvable groups
\cite{DFo3, To_Smallness}.
In particular, if $N$ is a free nilpotent
group of class $c \ge 2,$ then its automorphism group $\aut N$
is complete provided that $3 \le \rank(N) < \infty$
and $c=2$ \cite{DFo2}, or if $N$ is infinitely
generated \cite{To_2step, To_Aut(N)}.

Furthermore,
given a free group $F$
of countable rank at least three,
the \emph{outer} automorphism group $\out F$
of $F$ is complete \cite{BrVo, Khr, To_Out(F)}.
Essentially,
it has been demonstrated
that some structures recoverable from the group
$\out F$ with the use of
elements of finite order \cite{Khr, To_Out(F)} are ``sufficiently symmetric'', in line
with Baumslag's conjecture.  Importantly, this also applies to the standard
object on which the group $\out F$ acts in the case when $F$ is finitely
generated, namely the Outer space \cite{BrVo}.

In fact, the main technical result of
\cite{To_Out(F)} states that if $F$
is infinitely generated, then the group
$\out F$ is complete provided that
$F$ has the small index property
for free algebras \cite{To_Smallness4Nilps}.
Let $\mathfrak V$ be a variety
of algebras and let $\cF$
be a free algebra of infinite
rank from $\mathfrak V.$ We
say that $\cF$ has the {\it small index
property}, if any subgroup
$\Sigma$ of the automorphism
group $\Gamma=\aut\cF$ of $\cF$ having index at most $\rank(\cF)$
(a `small' one) contains
the pointwise stabilizer
$\Gamma_{(U)}$ of a subset of $\cF$
of cardinality $< \rank(\cF).$
Speaking of an arbitrary infinitely generated
free group $F,$ the small index
property for $F$ has been justified
only in the case when the rank of $F$
is countable \cite{BrEv}, and the general case
appears to be very challenging.
Another possible approach to settle the problem
of completeness of the group $\out F$
is to settle first the problem of obtaining a `manageable' generating
set of the group $\aut F,$ but it is (also)
a long-standing open problem.

Nevertheless, it seems very likely that
in the majority of known cases where
the automorphism group of a relatively free group
is complete, so is its outer automorphism group.
With this goal in mind, we endeavor
in the present paper to prove the following
result.

\begin{customthm}{\ref{MainTh}}
The outer automorphism group of an infinitely
generated free nilpotent group of class two
is complete.
\end{customthm}

Here, infinitely generated free nilpotent groups
are ideal testing ground: they have
the small index property \cite{To_Smallness4Nilps}
and their automorphism groups have nice
generating sets \cite{To_Berg}, and in the class two case
the technical difficulties are minimal.
Let $N$ be an infinitely generated
free nilpotent group of class two. It is relatively straightforward to show
that the group $\out N$ is centerless (Corollary
\ref{Centerless}), and the main challenge is to show that
all automorphisms of the group $\out N$
are inner. To the best of the author's knowledge,
the problem of the completeness of the outer
automorphism groups of finitely generated free nilpotent
groups of class at least two has not been addressed.

Principally, the main steps in the proof
Theorem \ref{MainTh} are quite similar
to the main steps in the proof
given in \cite{To_Out(F)} to show that
the outer automorphism group of a free
group of countably infinite rank is complete.
However, a possibility to recover in
the group $\out N$ the automorphism
group $\aut{A}$  of the abelianization
$A=N/[N,N]$ of $N$ (Proposition \ref{hatIA_is_def}) makes things considerably simpler.

\section{$\iat$-automorphisms and A-symmetries }

Everywhere below $N$ stands for an infinitely generated free nilpotent group of nilpotency class $2.$
Observe that the commutator subgroup $N'=[N,N]$ is
equal the center $\Za(N)$ of the group $N,$ and it is a free abelian group.
Let $\cX$ be a basis
of the group $N$ and let $<$
be a linear order on $\cX$
(here and below, we always
use the term `basis' to refer to a free generating set
of $N,$ \cite[pp. 9--14]{HN}).
It is easy to see that the commutator subgroup $N'$ is freely generated by
a subset of the set of all basis commutators
$[x,y]=xyx\inv y\inv,$ where $x,y \in \cX,$ with respect to
the basis $\cX,$ namely by the subset
$$
\{[x,y] : x,y \in \cX, x < y\}.
$$

An {\it \iat-automorphism} $\alpha$ of the group $N$
preserves all elements of the group $N$ modulo the subgroup
$N',$ or, in other words, the action of $\alpha$ on the
basis $\cX$ is given by
$$
\alpha x = x c_x, \qquad (x \in \cX).
$$
where $c_x \in N'.$
Accordingly, due to the fact that all $c_x$ are central elements of $N,$
where $x$ runs over $\cX,$
we have that $\alpha$ fixes pointwise all basis commutators
with respect to the basis $\cX,$
$$
\alpha( [x,y]) = [x c_x, y c_y]=[x,y], \qquad (x,y \in \cX),
$$
and hence $\alpha$ fixes each element of the commutator
subgroup $N'$ of $N.$ It then follows that the group
$\ia N$ of all \iat-automorphisms of $N$ is
a torsion-free abelian group.

Given any $z \in N,$ we shall denote by $\tau_z$
the inner automorphism
$$
\tau_z(a) =zaz\inv, \qquad (a \in N)
$$
of the group $N$ determined by $z.$ We shall
use the standard notation $\inn N$
for the group of all inner automorphisms
of $N.$ Clearly,
\begin{equation} \label{tau_z=tau_w}
\tau_z = \tau_w \iff z \equiv w \Mod{N'}, \qquad (z, w \in N).
\end{equation}
We shall use the following statement a number of times below.

\begin{Claim} \label{iaSqEqTauz}
Let $\tau_z,$ where $z \in N \setminus N',$ be a non-trivial
inner automorphism of the group $N$ and let
$$
\alpha^2 = \tau_z,
$$
where $\alpha \in \ia N.$ Then $\alpha=\tau_u$ is an inner
automorphism of the group $N$ determined by a suitable
$u \in N.$
\end{Claim}

\begin{proof}
Let $\cX$ be a basis of the group $N.$ Due to
(\theequation), we can assume that
$$
z = x_1^{m_1} x_2^{m_2} \ldots x_s^{m_s},
$$
where $x_i \in \cX$ are pairwise distinct
and $m_i \in \Z.$ Suppose that
$$
\alpha x = x c_x,
$$
where $c_x \in N'$ for all $x \in \cX.$

Now let $x \in \cX$ be arbitrary. By the conditions,
$$
\alpha^2 x = x c_x^2 = c_x^2 x = z x z\inv =\tau_z(x),
$$
and hence
\begin{equation}
c_x^2 = zxz\inv x\inv=
[z,x] = [\prod_{i=1}^s x_i^{m_i},x] = \prod_{i=1}^s [x_i,x]^{m_i}.
\end{equation}
Assume first that $x \notin \{x_1,x_2,\ldots,x_s\}.$ Then
the set
$$
[x_1,x], [x_2,x], \ldots, [x_s,x]
$$
of $\cX$-basis commutators can be extended to
a basis of the free abelian group $N'.$
Accordingly, the element $c_x$ must be
a product of the basis commutators $[x_i,x],$
that is,
$$
c_x = \prod_{i=1}^s [x_i,x]^{l_i},
$$
where $l_i$ are suitable integer numbers.
By (\theequation), we obtain that $m_i=2l_i$
for all $i,$ and so
denoting by
$u$ the product
$\prod_{i=1}^s x_i^{l_i},$ we see that
$$
c_x=[u,x].
$$
In the case when $x=x_i$ for some $i,$ we obtain,
arguing as before, that
$$
c_x = \prod_{\substack{i=1\\ x_i \ne x}}^s [x_i,x]^{l_i},
$$
but then clearly
$$
c_x=\prod_{i=1}^s [x_i,x]^{l_i}=[u,x],
$$
and the result follows.
\end{proof}

Everywhere below,
the letter $A$ will always denote the abelianization $N/[N,N]$ of the group $N.$
As it is traditional, the symbol $\av{\phantom a}$
will be used both for the natural homomorphism
$N \to A$ and for the induced homomorphism
$\aut N \to \aut A.$

The natural homomorphism $\aut N \to \out N=\aut N/\inn N$ will be denoted by $\widehat{\phantom a}.$
Given any element $s \in \out N,$ any $\sigma \in \aut N$ such that
$$
\widehat \sigma =\sigma \inn N = s
$$
will be said to {\it induce} $s.$
It is convenient to extend the use of the symbol
$\av{\phantom a}$ to denote the natural homomorphism
$\out N \to \aut A$ as well. The neutral element
of the group $\out N$ will be denoted simply by $1.$

As in the papers \cite{To_Towers, To_2step, To_Aut(N)}, an automorphism $\theta$ of the group $N$ which inverts
all elements of a certain basis of the group $N$ is called a {\it symmetry};
for simplicity's sake, an element $t \in \out N$ which is induced
by a symmetry in $\aut N$ will also be referred
to as a symmetry. Clearly, any symmetry in the group $\aut N$
(resp. in the group $\out N$) induces
the automorphism $-\id_A$ of the group $A.$

More generally, an element
$\theta \in \aut N$
will be called an A-{\it symmetry} if $\theta$ induces the
automorphism $-\id_A$ of the
\underline{a}belianization
$A$ of $N.$ As before,
the term will be shared with the elements
in the group $\out N$ that are induced
by A-symmetries in the group $\aut N.$

\begin{Lem} \label{Symmies:Basics}
{\rm (i)} Let $\theta$ be an \At-symmetry in
the group $\aut N.$ Then
$$
\theta \alpha \theta\inv = \alpha\inv
$$
for every \iat-automorphism $\alpha$ the
group $N.$

{\rm (ii)} Every \At-symmetry in the group
$\out N$ is an involution.
\end{Lem}

\begin{proof} (i) By part (a) of Lemma 2.1 in \cite{To_2step}.

(ii) Suppose that $t \in \out N$ is induced by
an automorphism $\theta \in \aut N$ and
$$
\av t = \av \theta = -\id_A.
$$
Then there is a symmetry $\theta^*$ in the
group $\aut N$ and an \iat-automorphism $\alpha \in \aut N$ such that
$$
\theta  =\theta^* \alpha.
$$
By (i),
$$
\theta^2 = \theta^* \alpha \theta^* \alpha
=\theta^*{}^2 \alpha\inv \alpha = \id_N,
$$
whence
$$
t^2 = (\widehat \theta)^2 = \widehat{\theta^2}=1.
$$
\end{proof}

\begin{Cor} \label{Centerless}
The group $\out N$ is centerless.
\end{Cor}

\begin{proof} Suppose $r$ is a central element
of the group $\out N.$ Then for every $s \in \out N,$
$$
s \cdot r = r \cdot s \To \av s \cdot \av r = \av r \cdot \av s,
$$
and this implies
that $\av r$ is central
in $\aut A,$ because the natural
homomorphism $\aut N \to \aut A$ and
hence the induced homomorphism $\out N \to \aut A$
are surjective \cite[Th. 5a]{Malt}.

Accordingly, $\av r = \pm \id_A,$ and so
$r$ is either induced by an IA-automorphism
of $N,$ or $r$ is an \At-symmetry in the group
$\out N.$

Take an IA-automorphism $\alpha \in \aut N$
such that $\widehat \alpha \ne 1$ and an \At-symmetry
$\theta \in \aut N,$ and suppose that
$$
\widehat \alpha \cdot \widehat \theta \cdot \widehat \alpha\inv =\widehat \theta.
$$
Then there is $z \in N$ satisfying
$$
\alpha \theta \alpha\inv = \tau_z \theta.
$$
By Lemma \ref{Symmies:Basics} (i), the
left-hand side of the last equality
is equal to $\alpha^2 \theta,$ and hence $\alpha^2 = \tau_z.$
By Claim \ref{iaSqEqTauz}, $\alpha = \tau_u,$ where
$u \in N,$ is an inner automorphism, which implies
that $\widehat \alpha=1,$ a contradiction. This demonstrates that neither $\widehat\alpha,$
nor $\widehat \theta$ can be a central element
of the group $\out N.$
\end{proof}

Let $G$ be a group.
Write $\cM_1(G)$ for
$G$ and let, for every natural number $k \ge 1,$ $\cM_{k+1}(G)$ be the family
of all $n$-ary relations on $\cM_k(G),$ where $n$ runs over $\N.$ Observe that the automorphism
group $\aut G$ acts in a natural way on all sets $\cM_k(G).$
When describing automorphisms of $G,$ it is often
necessary to show that a certain relation $R \in \cM_k(G)$
is invariant under all automorphisms of $G.$ This is the
case when the relation $R$ is the only realization of a
formula in a certain logic (e.g. the first-order, monadic second-order,
full second-order logic, etc.), and this formula is said to
define $R$ over $G.$ For convenience, we shall then say that
$R$ is \emph{definable} over $G.$
For stylistic purposes, in cases where a subset $S$
is definable over $G,$ we shall also say that $S$ is a \textit{definable subset} of
$G.$

In the case when a relation $R \in\cM_k(G)$
can be uniquely characterized by a formula of a certain
logic involving other relations $\{R_i : i \in I\}$
over $G,$ we shall say, extending the above agreement, that
$R$ is \emph{definable over $G$ with the parameters}
in the set $\{R_i : i \in I\}.$ Clearly, $R$ is then
invariant under any automorphism of $G$ under which all the relations
$R_i$ are invariant. For instance, if $X=\{x_i : i \in I\}$
is any subset of $G,$ the family
$\{\Za_G(x_i) : i \in I\} \in \cM_3(G)$
of all centralizers of the elements $x_i$ is definable over $G$
with parameters in $X.$

The reader is referred
to the paper \cite{To_2step} to find examples of definability
of relations over the group $\aut N$ by
means of the first-order and monadic second-order
logic. We should stress, however, that this
paper does not assume familiarity with the model theory;
whenever necessary, the reader can substitute
their own arguments to see that a particular
relation $R \in \cM_k(G)$ is preserved by all automorphisms
of $G=\out N.$ The highest order $k$
of the relations $R \in \cM_k(G)$ to be
considered below is equal to $4$
(as we shall need to characterize a set
of sets of subgroups of $G$ in Lemma \ref{def-o-A-bases}).

\begin{Prop}
An element
$t$ in the group $\out N$ is an \At-symmetry
if and only if $t$ is an involution
and any product of three conjugates
of $t$ is again an involution.
Consequently, the set of \At-symmetries is a definable
subset of the group $\out N.$
\end{Prop}

\begin{proof} Note that, by part (c) of Lemma 2.1 in
\cite{To_2step}, the same group-theoretic condition,
say $\chi(v),$ where $v$ is a variable, defines \At-symmetries in the group $\aut N.$
The proof is based on the fact that an involution
$\sigma$ in the group $\aut A$ satisfies the condition
$\chi(v)$ if and only if $\sigma=-\id_A.$

Suppose that an involution $t \in \out N$
satisfies the condition $\chi(v).$ Then
there is an automorphism $\theta$ of $N$ which
induces $t$ such
that
$$
\theta^2  \in \inn N
$$
and for every $\s_1,\s_2,\s_3 \in \aut N$
$$
( \sigma_1 \theta \sigma_1\inv \cdot \sigma_2 \theta \sigma_2\inv
\cdot \sigma_3 \theta \sigma_3\inv )^2 \in \inn N.
$$
It follows that $\av\theta$ satisfies the condition $\chi(v)$ in the group $\aut A,$
whence
$$
\av t =\av \theta=-\id_A,
$$
or, in other words, $t$ is an \At-symmetry.

Conversely, if $t \in \out N$ is an \At-symmetry,
then $\av t=-\id_A,$ and given any conjugates
$t_1,t_2,t_3$ of $t,$ we have that
$$
\av t_1 \cdot \av t_2 \cdot \av t_3 =-\id_A.
$$
Therefore, $t_1 t_2 t_3$ is an \At-symmetry, and thus,
by Lemma \ref{Symmies:Basics}, an involution in the group $\out N.$
\end{proof}

\begin{Prop} \label{hatIA_is_def}
The image $\widehat{\iat(N)} \sle \out N$ of the group
$\ia N$ is a definable subgroup of the group $\out N.$
\end{Prop}

\begin{proof}
Any element in the group $\widehat{\iat(N)}$
is a product of two A-symmetries, and vice versa.
\end{proof}

\section{Symmetries and extremal involutions}

An involution $\sigma$ in the group $\aut A$
is called {\it extremal} if there is a basis
$\cB = \{x\} \cup \cY$ of the group $A$
such that $\sigma$ sends $x$ to the
inverse element
and fixes the set $\cY$ pointwise.

As in \cite{To_Towers, To_2step, To_Aut(N)}, we will use the same term for an involution
$\f$ in the group $\aut N$ for which there
is a basis of the group $N$ such that
$\f$ inverts an element of this basis,
and fixes all other elements. As before,
an involution $f \in \out N$ will be called
{\it extremal} if it is induced by an extremal
involution in the group $\aut N.$

An involution $\f$ in the group $\aut N$
(resp. an involution $f$ in the group $\out N$)
will be called {\it \At-extremal} if
the image $\av\f$ of $\f$ (resp. $\av f$ of $f$) in the group
$\aut A$ is an extremal involution.

\begin{Lem} \label{Exts:Basics}
Let $\f \in \aut N$ be an extremal
involution and let
$\cB =\{x\} \cup \cY$ be a basis
of the group $N$ such that $\f x =x\inv$
and $\f y=y$ for all $y \in \cY.$

{\rm (i)} Suppose an element $t \in N \setminus N'$
satisfies
$$
\f t \equiv t\inv \Mod{N'}.
$$
Then $t = x^m c$ for a suitable integer
number $m$ and a suitable element $c \in N'$
of the commutator subgroup $N'.$

{\rm (ii)} The centralizer of $\f$ in the
group $\aut N$ consists of the automorphisms
of the form $\tau_v \s_0,$ where $v \in \str{\cY}$
and where $\s_0$
is an automorphism that
fixes each of the subgroups $\str x$ and $\str \cY$ setwise.

{\rm (iii)} The group $K=\ia N$ can be decomposed
as
$$
K = K^+_\f \cdot K_\f^-,
$$
where
$$
K^+_\f = \{ \alpha \in K : \f \alpha \f = \alpha\}
$$
and
$$
K^-_\f = \{ \alpha \in K : \f \alpha \f = \alpha\inv\}.
$$
\end{Lem}

\begin{proof}
(i) Write $t$ as
$$
t = x^m z c,
$$
where $z \in \str \cY$ and $c \in N'.$ Using
the action on the abelianization, we see that, by the conditions,
$$
\av \f (\av t) = - \av t,
$$
whence
$$
\av\f(m \av x + \av z) = -m \av x - \av z,
$$
or
$$
-m \av x + \av z = -m \av z - \av z,
$$
which implies that $\av z= 0,$ as required.

Before moving on to the proof of the next part,
observe that if $[y_1,y_2]$ is a basis commutator
of the elements $y_1,y_2 \in \cY,$ then
$$
\f ( [y_1,y_2] )  = [y_1,y_2].
$$
On the other hand, if $y$ is in $\cY,$ then
$$
\f( [x,y] ) = [x\inv, y] = [x,y]\inv.
$$
Thus $\f$ either fixes, or inverts
every basis commutator with respect to the
basis $\cB.$ It then follows that every element $d$ of the
commutator subgroup $N'$ of $N$ can be uniquely
written
as
$$
d = d^+ d^-
$$
so that $\f$ fixes $d^+$ and inverts $d^-.$

(ii) Suppose that $\s \in \aut N$ commutes with
$\f$: $\s \f = \f \s.$ Then
$$
\s \f x = \f \s x \To \f(\s x) =(\s x)\inv.
$$
By (i), $\s x =x^m c,$ where $m \in \Z$ and
$c \in N'.$ Clearly then, $m=\pm 1.$ On the other
hand,
$$
\f( x^m c ) =  c\inv x^{-m}
$$
implies that $\f(c)=c\inv.$ The above analysis of the action
of $\f$ on the basis commutators
with respect to the basis $\cB=\{x\} \cup \cY$
then shows that
$$
c=[x,u] =xux\inv u\inv=[x\inv,u\inv]
$$
for a suitable $u \in \str\cY.$ Hence if $m=1,$
we have \stepcounter{equation}
\begin{equation*} \tag{\theequation{}a}
\s x = x c =x \cdot x\inv u\inv x u=u\inv x u=\tau_{u\inv}(x),
\end{equation*}
and if $m=-1,$
\begin{equation*} \tag{\theequation{}b}
\s x = x\inv c = x\inv \cdot xux\inv u\inv = u x\inv u\inv=\tau_u(x\inv).
\end{equation*}
Now let $y \in\cY.$ Then commutativity of $\f$ and $\s$
implies that
$$
\f (\s y) =\s y.
$$
As before, let us analyze when an element of the
form $x^k z d,$ where $k \in \Z,$ $z \in \str\cY$
and $d \in N',$ is fixed by $\f$: clearly,
$$
\f( x^k z d ) = x^k z d \iff x^{-k} z \f(d) = x^k z d
$$
holds if and only if $k=0$ and $\f(d)=d;$ the latter
condition holds if and only if $d \in \str{\cY}.$
We therefore see that $\s$ must preserve
the subgroup $\str{\cY}.$ By (\theequation{}a)
and (\theequation{}b),
in the case when $m=1,$ the automorphism $\tau_u \s$
fixes $x,$ and in the case when $m=-1,$ the automorphism
$\tau_{u\inv} \s$ inverts $x.$ In each of the cases,
the corresponding automorphism $\s_0=\tau_{u^{\pm 1}} \s$ also fixes
the subgroup $\str{\cY}$ setwise, and we are done.

(iii) Consider an arbitrary IA-automorphism $\alpha$
whose action on the basis $\{x\} \cup \cY$
is given by
\begin{alignat*} 2
\alpha x &= x c_x        &&  \\
\alpha y &= y c_y, \quad && (y \in \cY).
\end{alignat*}
Now if
\begin{alignat*} 2
\alpha_1 x &= x c_x^-        &&  \\
\alpha_1 y &= y c_y^+, \quad && (y \in \cY)
\end{alignat*}
and
\begin{alignat*} 2
\alpha_2 x &= x c_x^+        &&  \\
\alpha_2 y &= y c_y^-, \quad && (y \in \cY),
\end{alignat*}
it is easy to verify that
$$
\alpha = \alpha_1 \alpha_2
$$
and $\alpha_1 \in K^+_\f$ and $\alpha_2 \in K^-_\f.$
\end{proof}

For the sake of notational simplicity,
as it has been done in the proof we have just completed,
we shall use the letter $K$ to denote the group $\ia N$ of
all \iat-automorphisms of the group $N.$

Let $X$ be a set and let $\Sigma \le \mathrm{Sym}(X)$
be a subgroup of the symmetric group of $X.$
Then, as is traditional in the theory of permutation
groups, given a subset $U$ of $X,$
we shall denote by $\Sigma_{(U)}$ the pointwise
stabilizer of the set $U$ in the group $\Sigma$
and by $\Sigma_{\{U\}}$ the setwise stabilizer
of $U$ in $\Sigma.$

We need the following result which is an immediate
corollary of Lemma 2.1 in \cite{MN}.

\begin{Lem} \label{MN-gen-sym}
Let $\Pi=\mathrm{Sym}(X)$ be the symmetric group of an infinite
set $X$ and let $U,V$ be disjoint subsets of $X$ both having
cardinalities less than $|X|.$ Then
$$
\Pi=\str{\Pi_{(U)}, \Pi_{(V)}}.
$$
\end{Lem}

Everywhere below $\Gamma$ stands for the automorphism
group $\aut N$ of $N.$

\begin{Prop} \label{headStabs}
{\rm (i)} The set of all \At-extremal involutions
in the group $\out N$
is a definable subset of this group.

{\rm (ii)} Let $f \in \out N$ be an \At-extremal
involution and let $x \in N$ be a primitive element
such that $\av f(\av x)=-\av x.$ Then the group
$$
G(f) = \widehat{\Gamma_{(x)} K} =
\{s \in \out N : \av s(\av x) = \av x\}
$$
is definable in the group $\out N$ with the parameter
$f.$
\end{Prop}

\begin{proof}
(i) By Lemma 1.4 in \cite{To_FreeAb}, the extremal
involutions form a definable subset
of the group
$\aut A,$ whence, by Proposition \ref{hatIA_is_def},
the result follows.

(ii) By Lemma 1.2 in \cite{To_Aut(A)}, the stabilizer
$\av \Gamma_{(\av x)}$ is definable in the
group $\av\Gamma=\aut A$ with any extremal involution
sending $\av x$ to the inverse element as a parameter,
and thus with the parameter $\av f.$
Apply then Proposition \ref{hatIA_is_def} again
to complete the proof.
\end{proof}

Let $\{f_i : i \in I\}$ be a family of \At-extremal
involutions from the group $\out N,$ whose
images in the group $\aut A$ are pairwise
distinct. For each
$i \in I,$ take a primitive element $x_i \in N$
satisfying $\av f_i(\av x_i)=-\av x_i.$
Consider the family $\cG=\{G(f_i) : i \in I\},$
where the subgroups $G(f_i)$ are defined
in
Proposition \ref{headStabs}.
We shall call $\cG$
an {\it  \At-basis modelling ensemble} over the group
$\out N$ if the set $\{\av x_i : i \in I\}$
is a basis of the group $A.$

\begin{Lem} \label{def-o-A-bases}
The set of all \At-bases modelling ensembles
is a definable object over the group $\out N.$
\end{Lem}

\begin{proof} Let $\cF=\{f_i : i \in I\}$ be a
family of \At-extremal involutions in the group
$\out N$ such that the involutions $\av f_i$
are pairwise distinct. For every $i \in I,$ choose a primitive element $x_i \in N$ such that
$\av f_i (\av x_i) = -\av x_i.$

A natural (classical) condition to impose
on the family $\cF$ is that
the elements of the family $\cF$
are pairwise commuting modulo the subgroup $K=\ia N.$
If so, it is easy to see that the set
$$
\cD=\{\av x_i : i \in I\}
$$
is a linearly independent subset of the group $A$
(as explained in the third paragraph on page 3162 in \cite{To_Aut(A)}).

To ensure that the set $\cD$ generates the group
$A$ we work as follows. Take an arbitrary
\At-extremal involution $g \in \out N$ and a primitive
element $x \in N$ satisfying $\av g(\av x)=-\av x.$
We then require that
there be finitely many indices $i_1,i_2,\ldots,i_s \in I$
such that
$$
G(g) \ge G(f_{i_1}) \cap G(f_{i_2}) \cap \ldots \cap G(f_{i_s}).
$$
The argument in the first paragraph on page 3163
in \cite{To_Aut(A)} then shows that
$$
\av x \in \str{\av x_{i_1}, \av x_{i_2}, \ldots, \av x_{i_s}}.
$$
Therefore, $\cD$ is a basis of the group $A.$
So $\{G(f_i) : i \in I\}$ is an A-basis
modelling ensemble provided that the conditions we have
considered in the process of the proof are true.
\end{proof}

We require the following concept
of the normalizer
of a family of subgroups.  For a group $G$
and a family $\mathcal{H} = \{H_i : i \in I\}$
of subgroups of $G,$ the \emph{normalizer} of
$\mathcal{H}$ in $G$ is the subgroup of
$G$ consisting of all elements $g$ such that
$$
g \cH g\inv =\{g H_i g\inv : i \in I\}=\cH
$$

\begin{Prop} \label{def_o_symms}
Let $\cG=\{G_i : i \in I\}$ be an \At-basis
modelling ensemble over the group $\out N.$

{\rm (i)} Consider an \At-symmetry $t \in \out N$
which commutes with each element
of the normalizer of the family $\cG$ in $\out N$ modulo
the subgroup $\widehat K^2.$ Then $t$ is a symmetry in
the group $\out N,$ and a suitable symmetry
$\theta \in \aut N$ in the preimage
of $t$ acts canonically---by inverting
all elements---on some basis of the group $N$
associated with the ensemble $\cG.$

{\rm (ii)} Consequently, the symmetries form a definable
subset
of the group $\out N.$
\end{Prop}

The condition in part (i) is reminiscent
of a similar condition
involving commutativity modulo the group
$$
K^2=\iat^2(N)=\{ \alpha^2 : \alpha \in K\}
$$
that defines the symmetries in the group
$\aut N$ (see Lemma 3.3 of \cite{To_2step}).

\begin{proof} (i) For every $i \in I,$ choose a primitive
element $x_i \in N$ such that
$$
G_i = \widehat{\Gamma_{(x_i)} K}.
$$
As $\cG$ is an A-basis modelling
ensemble, $\{\av x_i : i \in I\}$ is a basis
of the group $A,$ and hence the set $\cX=\{x_i : i \in I\}$
is a basis of the group $N$ (\cite[Th. 5a]{Malt}, \cite[31.25]{HN}).

Suppose that $t =\widehat \theta$
and write $\Pi \le \aut N$ for the group of all $\cX$-basis
permutations, that is, for the family
of automorphisms of $N$ which fix $\cX$ setwise;
observe that $\Pi \cong \sym\cX.$
Clearly,
each element of $\widehat \Pi$ fixes
the family $\cG$ setwise under the conjugation
action.

Then, by the conditions, for every $\pi \in \Pi,$
there exists $z \in N$ such that
\begin{equation}
\pi \theta \pi\inv \equiv \tau_z \theta \Mod{K^2}.
\end{equation}
From a broader perspective, the group $\Pi$
then acts by conjugation on the family
of cosets
$$
T=\{ \tau_w \theta K^2 : w \in N\}
$$
of the normal subgroup $K^2$ of the group $\aut N.$
Clearly, $|T| \le |N| = \rank(N),$ and every
$\Pi$-orbit in $T$ is of cardinality at
most $|\cX|=\rank(N).$ In particular, the index
$|\Pi : \mathrm{St}( \theta K^2 )|$
of the stabilizer $\mathrm{St}( \theta K^2 )$
of the coset $\theta K^2$ in the group $\Pi$
is at most $\rank(N)=|\cX|.$ By Theorem
$2^{\mbox{\normalsize$\flat$}}$ of \cite{DiNeuTho},
there is therefore a subset $U$ of $\cX$ of cardinality
less than
$|\cX|$
such that
$$
\mathrm{St}(\, \theta K^2\, ) \ge \Pi_{(U)}.
$$
Pick up a subset $V$ of cardinality $|U|$
in the complement $\cX \setminus U$ of $U,$ and consider
an involution $\rho \in \Pi$ which interchanges
$U$ and $V$ and fixes all other
elements of the basis $\cX.$
By Lemma \ref{MN-gen-sym}, $\Pi_{(U)}$
together with $\rho$ generates the group
$\Pi.$

According to (\theequation), there is $t \in N$
such that
\begin{equation} \label{action_of_rho}
\rho \theta \rho \equiv \tau_{t} \theta \Mod{K^2}.
\end{equation}
Suppose that $t \notin N'.$ Since $\tau_c = \id_N$ for every $c \in N',$
and since we are dealing with congruences
modulo the subgroup $K^2,$ we can assume
that $t$ is a product of pairwise
distinct letters $y_i,z_j$ from the basis $\cX,$
$$
t = \underbrace{y_1 \ldots y_{k}}_{a} \cdot
\underbrace{z_1 \ldots z_l}_{b},
$$
such that the first $k$ letters, whose product we shall denote
by $a,$ are in $U \cup V$ (if any)
and the rest of the letters, whose product
we shall denote by $b,$ are in $\cZ=\cX \setminus (U \cup V)$ (if any).
Suppose that $b \ne 1.$

Now, as $\cZ$ is infinite, there exists an $\cX$-permutational automorphism
$\s \in \Pi$ which fixes $U \cup V$ pointwise such that
\begin{equation}
\{z_1,\ldots,z_l\} \cap \s( \{z_1,\ldots,z_l\} ) = \varnothing.
\end{equation}
Clearly, $\s \in \Pi_{(U)} \le \mathrm{St}(\theta K^2)$ and $\s$ commutes with $\rho.$
Then conjugation of both sides of \eqref{action_of_rho}
by $\s$ produces,
as it is easy to see, the congruences
$$
\tau_t \theta \equiv \rho \theta \rho \equiv \tau_{\s(t)} \theta \Mod{K^2}
$$
Accordingly, $\tau_t \equiv \tau_{\s(t)} \Mod{K^2},$
whence
$$
\tau_{t\inv \s(t)} =\beta^2
$$
for a suitable IA-automorphism $\beta.$ By Claim \ref{iaSqEqTauz},
$\beta$ is an inner automorphism of the group $N,$ and we
obtain that the element
$$
t\inv \s(t) = b\inv a\inv \s( a b) =
b\inv a\inv a \s(b) = b\inv \s(b)
$$
is a square in the group $N$ modulo the commutator
subgroup $N',$ which is impossible in view
of (\theequation). Thus, we can assume that
$t \in \str{U \cup V}.$

Next, conjugating both sides of the congruence
\eqref{action_of_rho} by the involution $\rho,$ we
obtain that
$$
\theta \equiv \tau_{\rho(t)} \rho \theta \rho \equiv
\tau_{\rho(t)} \tau_t \theta \Mod{K^2},
$$
Arguing as above, we see that
the element $\rho(t) t$ is a square modulo the commutator subgroup $N'.$
We can assume that
$$
t = \underbrace{y_1 \ldots y_m}_c \cdot \underbrace{y_{m+1} \ldots y_k}_d,
$$
where the letters $y_i \in U\cup V$ are pairwise distinct,
$$
\rho(y_i) \not\in \cL=\{y_1,\ldots,y_{m}, y_{m+1},\ldots,y_k\}
$$
for all $i = m+1,\ldots,k$ and
\begin{equation}
\rho(\{ y_1, \ldots, y_m \}) = \{ y_1, \ldots, y_m \} = \cL \cap \rho(\cL),
\end{equation}
Suppose that $d \ne 1.$ In effect,
$$
\rho(\{y_{m+1},\ldots,y_k\}) \cap \{y_1,\ldots,y_{m}, y_{m+1},\ldots,y_k\}  =\varnothing,
$$
and hence the element $t \rho(t)$ cannot be a square
modulo the subgroup $N'.$ Thus $d =1,$ and in view
of (\theequation), we can assume that
the element $t$ is of the form
$$
t = \rho(u) u,
$$
where $u \in U$ is the product
of all elements of the set $U \cap \{ y_1, \ldots, y_m \}.$

Replacing $t$ with $\rho(u)u$ in \eqref{action_of_rho},
we therefore obtain that
$$
\rho \theta \rho \equiv \tau_{\rho(u) u} \theta \equiv
\tau_{\rho(u\inv)} \tau_u \theta,
$$
since evidently $\tau_{w\inv} \equiv \tau_w \Mod{K^2}$
for all $w \in N.$ It follows that
$$
\rho \tau_u \rho \cdot \rho \theta \rho
\equiv \rho \tau_u \theta \rho \equiv \tau_u \theta \Mod{K^2}.
$$
This implies that
$$
\pi \tau_u \theta \pi\inv \equiv \tau_u \theta \Mod{K^2}
$$
for {\it all} $\pi \in \Pi$ --- due to the fact
that $\Pi=\str{\Pi_{(U)},\rho}$ and due to
the fact
that all
elements of $\Pi_{(U)}$
fix $\tau_u$ under the conjugation action.

Now
denoting by
$\theta^* \in \aut N$ the
symmetry which inverts all elements of the
basis $\cX,$ we obtain, by Lemma 3.3 of \cite{To_2step}, that
$$
\tau_u \theta = \theta^* \alpha^2
$$
for an appropriate IA-automorphism $\alpha$ of $N.$
Evidently,
$$
\widehat{\tau_u \theta}=\widehat{\theta}=\widehat{ \alpha\inv \theta^* \alpha},
$$
and $\widehat \theta$ is a symmetry in the group
$\out N,$ as desired. Finally, as
$$
\alpha\inv \theta^* \alpha (\alpha\inv x_i) =\alpha\inv(x_i)\inv
$$
for all $i \in I$ and
$$
\av{\alpha\inv(x_i)}=\av x_i
$$
for all $i \in I,$ the basis $\{\alpha\inv(x_i) : i \in I\}$
is indeed one of the bases of $N$ associated with the
ensemble $\cG$ on which the symmetry $\theta^* \alpha^2$
acts canonically.

(ii)
The claim follows directly from (i).
\end{proof}

\begin{Prop} \label{Def_o_Exts}
Let an \At-basis modelling ensemble $\cG$ over
the group $\out N$ and a symmetry $t \in \out N$ satisfy the condition
in part {\rm (i)} of Proposition {\rm \ref{def_o_symms}}.

{\rm (i)} Suppose that an \At-extremal involution $f$
{\rm(}also{\rm)} normalizes the family
$\cG$ and commutes with $t.$ Then $f$
is an extremal involution in the
group $\out N.$

{\rm (ii)} Consequently, the extremal involutions
are definable in the group $\out N.$
\end{Prop}

\begin{proof} (i) Let
$\cG = \{G_i : i \in I\}.$
By Proposition \ref{def_o_symms}, there exists a basis
$\{x_i : i \in I\}$ of the group $N$ and
a symmetry $\theta \in \out N$ in the preimage
of $t$ such that for every
$i \in I,$ the element $G_i$ of the ensemble $\cG$ is of the form
$$
G_i =\widehat{\Gamma_{(x_i)} K}
$$
and $\theta x_i = x_i\inv.$
Next, choose an \At-extremal involution $\f \in \aut N$
in the preimage of $f.$ As $f$ normalizes the
family $\cG,$ we have that
$$
\f x_i \equiv x_i^{\pm 1} \Mod{N'}
$$
for all $i \in I.$ As $\f$ is an \At-extremal
involution, there is a $j \in I$ satisfying
$$
\f x_j \equiv x_j^{-1} \Mod{N'}
$$
and
$$
\f x_i \equiv x_i \Mod{N'}
$$
for all $i \in I \setminus \{j\}.$ Accordingly,
denoting by
$\f_0 \in \aut N$ the extremal involution
which inverts $x_j$ and fixes all other elements
of the basis $\{x_i : i \in I\},$ we have that
\begin{equation}
\f = \alpha \f_0,
\end{equation}
where $\alpha$ is an IA-automorphism of
the group $N.$ Clearly, $\f_0$ commutes
with $\theta.$

Since $f=\widehat \f$ and $t=\widehat \theta$ are commuting, we get that
$$
\f \theta \f\inv = \tau_z \theta
$$
for a suitable $z \in N,$ and then
$$
\alpha \f_0 \theta \f_0 \alpha\inv = \tau_z \theta,
$$
whence
$$
\alpha \theta \alpha\inv = \tau_z \theta,
$$
or
$$
\alpha^2 \theta = \tau_z \theta.
$$
Thus, $\tau_z =\alpha^2,$
which means that $\alpha=\tau_u$ is an inner automorphism
of the group $N$ determined by a suitable $u \in U.$
By (\theequation),
$$
f=\widehat \f = \widehat{\tau_u \f_0}=\widehat \f_0,
$$
and $f$ is indeed an extremal involution in $\out N,$ as claimed.

(ii)
The claim follows directly from (i).
\end{proof}

To simplify the notation,
we will denote simply by $\Za(*)$
the centralizer $\Za_{\aut N}(*)$
of a subset $*$ of the group $\aut N.$
The symbol $\Zo(*)$ will denote
the centralizer $\Za_{\out N}(*)$ of a subset $*$
of the group $\out N.$

\begin{Lem} \label{Ext:Z(f)}
Let $\f \in \aut N$ be an extremal involution and
let $\cB=\{x\} \cup \cY$ be a basis of the
group $N$ such that $\f x=x\inv$ and
$\f y=y$ for all $y \in \cY.$ Then the following
statements are true.

{\rm (i)} The centralizer of the image $\widehat \f$
of $\f$ in the group $\out N$ is induced by the centralizer
of $\f$ in the group $\aut N,$ or, in other words,
$$
\Zo(\widehat \f) = \widehat{\Za(\f)}.
$$

{\rm (ii)} The commutator subgroup
$\Zo(\widehat \f)'=[\Zo(\widehat \f),\Zo(\widehat \f)]$ of the group $\Zo(\widehat \f)$
is induced by the commutator subgroup
$\Za(\f)'=[\Za(\f),\Za(\f)]$ of the group $\Za(\f).$

{\rm (iii)} Every element
of the subgroup $\Zo(\widehat \f)'$ is induced
by a uniquely determined automorphism of the group $N$ which
preserves $x$ and fixes the subgroup
$\str{\cY}$ setwise, and vice versa. Accordingly,
$$
\Zo(\widehat \f)' \cong \aut N.
$$

{\rm (iv)} Let $z \in \cY$ and let $\pi \in \aut N$
be an involution which interchanges $x$ and $z,$
and fixes $\cY \setminus \{x,z\}$ pointwise.
Then the group $\out N$ is generated by
$\Zo(\widehat \f)'$ and $\widehat \pi.$
\end{Lem}

\begin{proof} (i) Take an automorphism
$\s \in \aut N$ outerly commuting with
$\f$:
$$
\widehat \s \cdot \widehat \f = \widehat \f \cdot \widehat \s.
$$
Then
$$
\sigma \f \sigma\inv = \tau_t \f
$$
for an appropriate $t \in N.$ Clearly,
$$
\id_N=(\sigma \f \sigma\inv)^2 = \tau_t \f \cdot \tau_t \f,
$$
and hence
$$
\tau_{\f(t)} = \tau_t\inv = \tau_{t\inv}.
$$
It follows that
$$
\f t \equiv t\inv \Mod{N'},
$$
and so, in light of (i) of Lemma \ref{Exts:Basics}, $t=x^m c$ for some
$m \in \Z$ and $c \in N'.$ Consequently,
$$
\tau_t = \tau_{x^m c} =\tau_{x^m} =\tau_x^m.
$$

Now suppose
that
$m$ is even: $m=2k,$ where $k \in \Z.$
Then
$$
\sigma \f \sigma\inv = \tau_x^{2k} \f  = \tau_x^k \f \tau_x^{-k},
$$
whence
$$
\tau_{x}^{-k} \sigma \cdot \f \cdot \sigma\inv \tau_x^k=\f.
$$
We then see that $\tau_x^{-k} \sigma$ commutes
with $\f,$ which means that
$$
\widehat \s=\widehat{\tau_x^{-k} \sigma}
$$
is indeed induced by an element of $\Za(\f),$
as required.

Next, let $m$ be odd: $m=2k+1,$ where $k \in \Z.$
Arguing as before, we obtain that
$$
\rho \f \rho\inv = \tau_x \f,
$$
where $\rho = \tau_x^{-k} \s.$
As $\av \f$ and $\av \rho$ commute in the
group $\aut A,$
$$
\av \rho(\av x)=\pm \av x
$$
and $\av \rho$ stabilizes the fixed-point
subgroup $\str{\av \cY}$ of $\av \f.$
This implies that
$$
\av \rho = \av \rho_0,
$$
for an appropriate automorphism $\rho_0$
of the group $N$ for which we have
$$
\rho_0 x = x^{\pm 1}
$$
and $\rho_0 \str{\cY} =\str{\cY}.$ Clearly,
$\rho_0$ commutes with $\f$ and
$$
\rho=\alpha \rho_0
$$
for some $\alpha$ in $K=\ia N.$

Therefore
$$
\alpha \rho_0 \f \rho_0\inv \alpha\inv = \tau_x \f,
$$
or
\begin{equation}
\alpha \f \alpha\inv = \tau_x \f,
\end{equation}
Write $\alpha$ as
$$
\alpha = \alpha_1 \alpha_2,
$$
where $\alpha_1 \in K^+_\f$ and $\alpha_2 \in K^-_\f$
(please consult Lemma \ref{Exts:Basics} if necessary).
Then we deduce from (\theequation) that
$$
\tau_x = \alpha \cdot \f \alpha\inv \f =
\alpha_1 \alpha_2 \cdot \alpha_1\inv \alpha_2 = \alpha_2^2,
$$
thereby obtaining that $\tau_x$ is a square
in the group $K,$ and hence $x$ is a square
in the group $N$ modulo the subgroup $N',$
which is impossible.

Finally, since obviously $\widehat{\Za(\f)} \le \Zo(\widehat \f)$
and the above argument demonstrates that
$\Zo(\widehat \f) \le \widehat{\Za(\f)},$ we get
that
\begin{equation}
\Zo(\widehat \f) = \widehat{\Za(\f)},
\end{equation}
as claimed.

(ii) Taking into account that the map $\widehat{\phantom a}$
is a homomorphism, we derive from (\theequation) that
$$
\Zo(\widehat \f)'=
[\Zo(\widehat \f),\Zo(\widehat \f)] =
[\widehat{\Za(\f)}, \widehat{\Za(\f)}]=
\widehat{ [\Za(\f), \Za(\f)]}=\widehat{\Za(\f)'}.
$$

(iii) Let $\rho_1,\rho_2 \in \Za(\f)$ be elements
of the centralizer of $\f$ in the group
$\aut N.$ By part (ii) of Lemma \ref{Exts:Basics},
$$
\rho_1 = \tau_{u_1} \s_1 \text{ and } \rho_2 = \tau_{u_2} \s_2,
$$
where $u_1,u_2 \in \str{\cY},$
$$
\s_1(x)=x^\eps \text{ and } \s_2(x) =x^\eta,
$$
where $\eps, \eta =\pm 1,$ and both $\s_1,\s_2$
preserve the subgroup $\str{\cY}.$ Clearly,
$$
\widehat{[\rho_1,\rho_2]}=\widehat{[\s_1,\s_2]}.
$$
Furthermore,
\begin{align*}
[\s_1,\s_2](x) &= \s_1 \s_2 \s_1\inv \s_2\inv(x)
= x
\end{align*}
and
$$
[\s_1,\s_2](\str\cY) =\str\cY.
$$
Accordingly,
$$
\widehat{\Za(\f)'} \le
\widehat C,
$$
where
$$
C = \{\lambda \in \aut N : \lambda(x) =x \text{ and } \lambda(\str\cY)=\str\cY\} \le \Za(\f).
$$
On the other hand, since evidently
$C \cong \aut N$ and since the latter
group is perfect \cite[Prop. 2.4, Th. 3.1]{To_Smallness4Nilps},
every element of $C$ is a product of commutators
of elements of $\Za(\f),$
whence $\widehat{\Za(\f)'}=\widehat C.$

It remains to show that
$$
\widehat \lambda_1 = \widehat \lambda_2 \To \lambda_1=\lambda_2
$$
for all $\lambda_1,\lambda_2$ in $C.$ This is immediate,
because the equality
$$
\tau_z \lambda_1 = \lambda_2
$$
for some $z \in N$ leads to $z = x^k c$ for suitable
$k \in \Z$ and $c \in N'$---due to the fact
that $\lambda_1(x)=\lambda_2(x)=x.$ The number $k$
must be equal to zero, since otherwise
$$
\tau_{x^k} \lambda_1(\str\cY) =x^k \str\cY x^{-k} \ne \str\cY=\lambda_2(\str\cY).
$$

(iv) Recall that a subset $J$ of an infinite
set $I$ is called a {\it moiety} of $I$
if $|J|=|I \setminus J|.$

An automorphism $\s$ of the group
$N$ is called a {\it $\cB$-moietous},
where $\cB$ is the basis given
in the conditions,
if there is a moiety
$\cC$ of $\cB$ such that $\s$ fixes setwise
the subgroup $\str\cC$ generated by $\cC$
and fixes the set $\cB \setminus \cC$ pointwise.
Theorem 2.5 from \cite{To_Berg} can be equivalently
rephrased by saying that the group $\aut N$ is generated
by all $\cB$-moietous automorphisms. It follows
that the group $\aut N$ is then generated by the subgroup
$\Za(\f)'$ and $\pi.$ Indeed, write $\Pi$ for the
set of all $\cB$-permutational automorphisms
of $N,$ that is, automorphisms that fix $\cB$
setwise. Clearly,
$\Pi_{(x)} =\Pi \cap \Za(\f)',$
or, in other words, a $\cB$-permutational
automorphism fixes $x$ if and only if this
automorphism is in $\Za(\f)'.$ Further, since
$$
\pi \Pi_{(x)} \pi =\Pi_{(z)},
$$
the group $\str{\Za(\f)',\pi}$ contains
both subgroups $\Pi_{(x)}$ and $\Pi_{(z)}$
which generate the subgroup $\Pi$ (Lemma \ref{MN-gen-sym}).
We therefore see that the group $\str{\Za(\f)',\pi}$
contains all $\cB$-permutational automorphisms,
and so all $\cB$-moietous automorphisms, and hence
$$
\aut N=\str{\Za(\f)',\pi}.
$$
The result then follows, since $\widehat{\phantom a} : \aut N \to \out N$
is a surjective homomorphism.
\end{proof}

\section{Completeness}

Having an element $r \in \out N,$ we shall denote
the inner automorphism
$$
s \mapsto r s r\inv, \qquad (s \in \out N)
$$
of the group $\out N$ determined by $r$
by $T_r.$

\begin{Th} \label{MainTh}
The outer automorphism group of any infinitely
generated free nilpotent group of class two is complete.
\end{Th}

\begin{proof}
As we have seen above,
the group
$\out N$ is centerless (Cor. \ref{Centerless}).
Consider an arbitrary automorphism $\Delta$ of the group
$\out N.$
We need to demonstrate that
$\Delta$ is an inner automorphism of the group $\out N,$
thus completing the proof.

We aim at multiplying $\Delta$ by a number of suitable
inner automorphisms of the group $\out N$ so that
the resulting automorphism is the trivial one.

Let $f$ be an extremal involution in the group
$\out N.$ Due to definability of the extremal
involutions in the group $\out N$ (Lemma \ref{Def_o_Exts}), the image
$\Delta(f)$ of $f$ is also an extremal involution.
Since the extremal involutions form a conjugacy
class in the group $\out N,$ there exists therefore
$r \in \out N$ such that
$$
(T_r \circ \Delta)(f) = f.
$$
Write $\Delta_1$ for the product $T_r \circ \Delta.$

Now as the automorphism $\Delta_1$ fixes
$f,$ it fixes setwise the centralizer $\Zo(f)$
and hence fixes setwise its commutator subgroup $\Zo(f)'$:
$$
\Delta_1(\, \Zo(f)'\,) = \Zo(f)'.
$$
Next, by Proposition \ref{Ext:Z(f)}, the group
$\Zo(f)'$ is isomorphic to the group $\aut N$
which is complete \cite[Th. 5.1]{To_2step}.
Accordingly, taking a suitable $s \in \Zo(f)',$
we see that the automorphism
$$
\Delta_2 = T_s \circ \Delta_1
$$
fixes pointwise all elements of the
group $\Zo(f)'$ and, obviously, fixes $f.$

Choose an extremal involution $\f \in \aut N$
which induces $f.$
Let $\cB = \{x\} \cup \cY$
be a basis of the group $N$ on which $\f$ acts
canonically, that is, $\f x=x\inv$ and
$\f(y)=y$ for all $y \in \cY.$

Consider then the symmetry $\theta$ which inverts
all elements of the basis $\cB.$ Write $\theta$ as
$$
\theta = \f \theta_0,
$$
where the automorphism $\theta_0$ of $N$ fixes
$x$ and inverts all elements of the set $\cY.$
By Lemma \ref{Ext:Z(f)}, $\widehat \theta_0$ is in
$\Zo(f)',$ and then
$$
\Delta_2( \widehat \theta)
=\Delta_2(\widehat \f \cdot \widehat\theta_0)
=\widehat \f \cdot \widehat \theta_0=\widehat\theta.
$$
Bearing in mind
that we intend to use part (iv)
of Lemma 3.7,
we need a $\cB$-permutational automorphism $\pi$ of
order two from the group $\aut N.$
Take an element $z \in \cB,$
and then choose $\pi$ to be the automorphism
of the group $N$ which interchanges $x$ and $z,$
and fixes $\cB \setminus \{x,z\}$ pointwise.
It is important to note that $\pi$ commutes with $\theta.$

Observe for future
use that given any
automorphism $\Upsilon$ of the group
$\out N,$ we have that
\begin{equation} \label{Ups_conj}
\av s \sim \overline{\Upsilon(s)}, \qquad (s \in \out N),
\end{equation}
where $\sim$ is the conjugacy
relation on the group $\aut A.$ This is due to the fact that the map
$$
\av s \mapsto \overline{\Upsilon(s)},
$$
where $s$ runs over the group $\out N,$
is a well-defined automorphism of the
group $\aut A$ (in view of definability
of the group $K$) and the fact that
all automorphisms of the group
$\aut A$ are inner \cite[Th. 2.1]{To_Aut(A)}.

Consider the \At-basis modelling ensemble
$\{G_b : b \in \cB\},$ where
$$
G_b =\widehat{\Gamma_{(b)}K}, \qquad (b \in \cB).
$$
Now for every $b \in \cB,$ the group $G_b$ is definable
in the group $\out N$ with the parameter $\widehat \f_b,$ where $\f_b$ is the extremal
involution which inverts $b$ and fixes all other
elements of $\cB;$ note that $f=\widehat \f=\widehat \f_x.$ As we have that
$$
\widehat \f_y \in \Zo(f)', \text{ and hence }
\Delta_2( \widehat \f_y )=\widehat \f_y
$$
for all $y \in \cY,$ then
\begin{equation}
\Delta_2( G_b ) = G_b
\end{equation}
for all $b \in \cB.$

The involution $\widehat \pi \in \out N$
induced by $\pi$ admits the following
description:
$$
\widehat \pi G_x \widehat \pi = G_z \text{ and }
\widehat \pi \in G_b
$$
for all $b \in \cB \setminus \{x,z\}.$ By (\theequation),
\begin{equation}
\Delta_2(\widehat \pi) G_x \Delta_2(\widehat \pi) = G_z \text{ and }
\Delta_2(\widehat \pi) \in G_b
\end{equation}
for all $b \in \cB \setminus \{x,z\}.$ Let
$$
\Delta_2( \widehat \pi ) = \widehat{\pi_1},
$$
where $\pi_1 \in \aut N.$ By (\theequation), either
$$
\pi_1 x = z c_z, \text{ or } \pi_1 x = z\inv c_z,
$$
where $c_z \in N'.$ Due to $\widehat \pi_1$
being an involution in the group $\out N,$ then either
$$
\text{(a): }
\begin{cases}
\pi_1 x = z c_z, \\
\pi_1 z = x c_x,
\end{cases}
\text{ or }
\text{(b): }
\begin{cases}
\pi_1 x = z\inv c_z, \\
\pi_1 z = x\inv c_x,
\end{cases}
$$
where $c_x \in N'.$ Furthermore,
by (\theequation),
$$
\pi_1 b = b c_b,
$$
where $c_b \in N',$ for all $b \in \cB \setminus \{x,z\}$
(for if there were any $b \in \cB \setminus \{x,z\}$
which $\pi_1$ sends to $b\inv$ modulo
$N',$ then $\av \pi_1$ would not be
conjugate to $\av \pi,$ contradicting
\eqref{Ups_conj}).

Accordingly, if (a) holds, then
$$
\pi_1 = \beta \pi,
$$
where $\beta \in K.$ Suppose that (b) holds. If so,
$$
\pi_1 z = x\inv c_x \To \pi_1 \f (z) = \f(x) c_z,
$$
and then
$$
(\f \pi_1 \f)(z) = x \f(c_z).
$$
Again, since $\widehat \pi_1$ is an involution,
we obtain, arguing as above, that
$$
\f \pi_1 \f = \beta \pi
$$
for some $\beta \in K.$

So let $\Delta_3 = \Delta_2$ if (a) holds,
and $\Delta_3 = T_f \circ \Delta_2$ if (b) holds.
Now, since
$\Delta_2,$ fixes $\widehat \theta$ and fixes $\Zo(f)'$
pointwise, it is clear that $\Delta_3$ does, too.
We claim that $\Delta_3$ fixes
$\widehat \pi$ as well---due to the fact that
$\pi$ commutes with $\theta.$
Indeed,
$$
\pi \theta \pi = \pi \theta \pi\inv= \theta
$$
implies that
$$
\Delta_3(\widehat \pi) \cdot \widehat \theta \cdot \Delta_3(\widehat \pi)\inv = \widehat\theta.
$$
As it has been discovered above, $\Delta_3(\widehat \pi)$
is induced by an automorphism $\beta \pi$ for some
$\beta \in K.$ Thus
$$
\beta \pi \theta \pi \beta\inv = \tau_v \theta
$$
for a suitable $v \in N.$ We then get that
$$
\beta^2 = \tau_v,
$$
whence $\beta=\tau_u,$ where $u \in N,$ is an inner automorphism
of the group $N.$ Consequently,
$$
\Delta_3( \widehat \pi ) = \widehat{\tau_u \pi}=\widehat \pi.
$$
We therefore conclude, by part (iv)
of Lemma \ref{Ext:Z(f)}, that $\Delta_3$ fixes
all elements of the group $\out N,$ and so it
is the trivial automorphism. By the construction,
the
original
automorphism $\Delta$ we have started with
is therefore an inner automorphism of the
group $\out N,$ as claimed.
\end{proof}

\end{document}